\documentclass[11pt]{article}

\usepackage[a4paper,margin=1in]{geometry}
\usepackage{amsmath,amssymb,amsthm,mathtools}
\usepackage{booktabs}
\usepackage{float}
\usepackage{xurl}
\usepackage{hyperref}
\usepackage[nameinlink,capitalize]{cleveref}

\hypersetup{colorlinks=true,linkcolor=blue,citecolor=blue,urlcolor=blue}

\newtheorem{theorem}{Theorem}[section]
\newtheorem{lemma}[theorem]{Lemma}
\newtheorem{proposition}[theorem]{Proposition}
\newtheorem{corollary}[theorem]{Corollary}
\newtheorem{remark}[theorem]{Remark}

\DeclareMathOperator{\tr}{tr}
\DeclareMathOperator{\KL}{KL}
\newcommand{\R}{\mathbb R}
\newcommand{\Spp}{\mathbb S_{++}}
\newcommand{\cN}{\mathcal N}
\newcommand{\cM}{\mathcal M}
\newcommand{\cP}{\mathcal P}
\newcommand{\dd}{\,\mathrm d}
\newcommand{\Id}{I}

\title{Closed Forms for Gaussian Kullback--Leibler Unbalanced Optimal Transport without Coupling Entropy}
\author{Jiaping Yang\thanks{School of Mathematical Sciences, Fudan University. Email: \texttt{jpyang22@m.fudan.edu.cn}}
 \and Yunxin Zhang\thanks{School of Mathematical Sciences, Fudan University. Email: \texttt{xyz@fudan.edu.cn}}}
\date{}

\begin{document}
\maketitle

\begin{abstract}
We obtain an explicit solution for the static Kullback--Leibler (KL) unbalanced optimal transport problem between finite non-degenerate Gaussian measures with quadratic cost, two independent positive marginal relaxation parameters, and no entropy penalty on the coupling.  The minimizer is a scaled Wasserstein coupling between two adjusted Gaussian marginals and is supported on an affine graph; in entropic Gaussian unbalanced transport, by contrast, the optimal plan is non-degenerate on the product space.  The covariance map is the unique positive definite solution of a Riccati equation and admits a principal-square-root representation.  Compared with the known equal-penalty Gaussian Hellinger--Kantorovich endpoint, the result treats the asymmetric two-sided Kullback--Leibler relaxation and gives the modified marginals, joint minimizer, value, and a direct quadratic KL-dual certificate.  The proof combines a Gaussian reduction with an endpoint duality argument, so the graph singularity is part of the construction.  The large-relaxation limit recovers the Gaussian Wasserstein cost for equal masses.
\end{abstract}



\section{Introduction}

Explicit Gaussian solutions are standard benchmarks in optimal transport and often reveal the underlying matrix geometry.  The best-known example is the quadratic Wasserstein formula, whose covariance term is the Bures--Wasserstein distance on the cone of positive definite matrices; see Dowson and Landau \cite{dowson_landau_1982}, Gelbrich \cite{gelbrich_1990}, Takatsu \cite{takatsu_2011}, and Bhatia, Jain and Lim \cite{bhatia_jain_lim_2019}.

In many optimization and data-analysis problems, the hard marginal constraints of balanced transport are too restrictive.  Unbalanced optimal transport (UOT) replaces them by divergence penalties, allowing mass to be removed or created.  The static and dynamic foundations of this viewpoint were developed in, among others, Chizat, Peyr\'e, Schmitzer and Vialard \cite{chizat_interpolating_2018,chizat_scaling_2018} and Liero, Mielke and Savar\'e \cite{liero_mielke_savare_2018}; see also Peyr\'e and Cuturi \cite{peyre_cuturi_2019} and S\'ejourn\'e, Peyr\'e and Vialard \cite{sejourne_2022} for computational perspectives.  More recent algorithms have emphasized unregularized or weakly regularized UOT formulations, including non-negative penalized regression reformulations \cite{chapel_2021} and scalable low-rank solvers \cite{scetbon_2023}.  This makes continuous test problems with tractable formulas useful for checking weakly regularized regimes.

The problem studied here belongs to the entropy-transport family in the sense that marginal discrepancies are penalized by KL divergences; the notation \(\KL(\rho\mid\eta)\) denotes the generalized relative entropy defined in \eqref{eq:KL} below.  With equal penalties and quadratic ground cost, it is closely related to the Gaussian Hellinger--Kantorovich distance \(GHK_\gamma\) discussed by Liero, Mielke and Savar\'e \cite{liero_mielke_savare_2018}.  In particular, Janati's thesis manuscript identifies the zero-coupling-entropy limit of entropic Gaussian UOT with \(GHK_\gamma\) and gives an explicit value for Gaussian measures in the symmetric case where the two marginal-penalty parameters satisfy \(\tau_0=\tau_1=\gamma\) \cite[Prop.~14]{janati_thesis_2021}.  We therefore treat the equal-penalty Gaussian value formula as prior work.  Our focus is the asymmetric two-sided KL relaxation with independent parameters \(\tau_0,\tau_1\), together with the adjusted Gaussian marginals, the graph-supported minimizer, and a direct primal-dual optimality certificate at the endpoint.

A second comparison is with strictly entropic Gaussian transport.  Janati, Muzellec, Peyr\'e and Cuturi \cite{janati_2020} derived explicit expressions for Gaussian entropic optimal transport and for scaled Gaussian variants; for input measures \(\alpha,\beta\), a transport plan \(\pi\), and with \(\otimes\) denoting product measure, their model contains a plan-level term
\[
    \varepsilon\KL(\pi\mid\alpha\otimes\beta),\qquad \varepsilon>0,
\]
possibly together with marginal KL penalties.  This positive coupling entropy enforces absolute continuity of the optimal plan with respect to \(\alpha\otimes\beta\), so the minimizer is a non-degenerate scaled Gaussian measure on the product space.  Related explicit Gaussian formulas have also been developed for the quadratic Wasserstein distance \(W_2\) with entropy regularization, Gaussian Schr\"odinger bridges, Gaussian Sinkhorn recursions, and entropic Gromov--Wasserstein variants \cite{mallasto_gerolin_minh_2022,bunne_2023,akyildiz_delmoral_miguez_2024,le_2022}.  All these models retain positive entropy or impose different constraints.  Once the plan-level entropy is removed, absolute continuity on the product space is no longer the correct structural ansatz: the optimal measure may be singular and supported on a transport graph.

We work directly at zero coupling entropy.  For finite non-degenerate Gaussian measures \(\alpha=a\mathcal N(m_0,\Sigma_0)\) and \(\beta=b\mathcal N(m_1,\Sigma_1)\), with \(a,b>0\), \(\Sigma_0,\Sigma_1\in\mathbb S_{++}^d\), and \(\tau_0,\tau_1>0\), we solve
\begin{equation*}
    \inf_{\gamma\in\mathcal M_+(\mathbb R^d\times\mathbb R^d)}
    \left\{
    \int \|x-y\|^2\,d\gamma(x,y)
    +\tau_0\KL(\gamma_0\mid\alpha)
    +\tau_1\KL(\gamma_1\mid\beta)
    \right\},
\end{equation*}
where \(\gamma_0,\gamma_1\) are the two marginals of \(\gamma\).  The resulting finite-dimensional expressions identify the transported mass, the modified Gaussian marginals, the affine map between them, and the optimal value.  The covariance component reduces to a positive definite Riccati equation with a principal-square-root solution.

The contribution is threefold.  First, we compute the solution for arbitrary positive, possibly unequal relaxation parameters \(\tau_0,\tau_1\), including the transported mass, adjusted marginals, graph-supported plan, and value.  Second, the Gaussian reduction separates the scalar mass, mean, and covariance variables; the covariance part is governed by a Riccati equation with a unique positive solution written through a principal square root.  Third, explicit quadratic KL-dual potentials certify global optimality and uniqueness by complementary slackness on the graph, without appealing to a singular limit of entropic densities.  The result gives a compact continuous benchmark for weakly regularized UOT regimes and clarifies how the balanced Gaussian Wasserstein formula is recovered when marginal relaxation becomes large.

\section{Closed-form result}

For finite non-negative measures \(\rho,\eta\), the convention is
\begin{equation}\label{eq:KL}
    \KL(\rho\mid\eta)
    =
    \begin{cases}
    \displaystyle\int (r\log r-r+1)\dd\eta, & r=\frac{\mathrm{d}\rho}{\mathrm{d}\eta},\\[2pt]
    +\infty, & \rho\not\ll\eta.
    \end{cases}
\end{equation}
Here \(\rho\ll\eta\) denotes absolute continuity.  We write \(\cM_+(E)\) for the cone of finite non-negative Borel measures on a measurable space \(E\), \(\cP_2(\R^d)\) for the set of Borel probability measures on \(\R^d\) with finite second moment, and \(W_2\) for the quadratic Wasserstein distance on \(\cP_2(\R^d)\).  We also write \(\Spp^d\) for the cone of real \(d\times d\) symmetric positive definite matrices and \(\Id\) for the \(d\times d\) identity matrix.  The notation \(B\succ0\) means that a symmetric matrix \(B\) is positive definite.  We use \(\tr\) and \(\det\) for trace and determinant, respectively.
Let \(\mu_0=\cN(m_0,\Sigma_0)\) and \(\mu_1=\cN(m_1,\Sigma_1)\), where \(\cN(m,\Sigma)\) denotes the Gaussian probability measure with mean \(m\) and covariance \(\Sigma\), \(m_0,m_1\in\R^d\), and \(\Sigma_0,\Sigma_1\in\Spp^d\).  Let \(\alpha=a\mu_0\), \(\beta=b\mu_1\), with \(a,b>0\).  For \(\gamma\in\cM_+(\R^d\times\R^d)\), its marginals are denoted by \(\gamma_0=(\mathrm{pr}_0)_\#\gamma\) and \(\gamma_1=(\mathrm{pr}_1)_\#\gamma\), where \(\mathrm{pr}_0(x,y)=x\), \(\mathrm{pr}_1(x,y)=y\), and \(F_\#\rho\) denotes the push-forward of \(\rho\) by \(F\).  The object of study is
\begin{equation}\label{eq:uot-primal}
    \mathsf U_{\tau_0,\tau_1}(\alpha,\beta)
    =
    \inf_{\gamma\in\cM_+(\R^d\times\R^d)}
    \left\{
    \int \|x-y\|^2\dd\gamma(x,y)
    +\tau_0\KL(\gamma_0\mid\alpha)
    +\tau_1\KL(\gamma_1\mid\beta)
    \right\}.
\end{equation}
Put
\begin{equation*}
    r_0=\frac{2}{\tau_0},\qquad
    r_1=\frac{2}{\tau_1},\qquad
    A_0=\Sigma_0^{-1},\qquad
    A_1=\Sigma_1^{-1},
\end{equation*}
\begin{equation*}
    C_0=A_0+r_0\Id,
    \qquad
    C_1=A_1+r_1\Id,
    \qquad
    \kappa=r_1-r_0.
\end{equation*}
All matrix square roots below are principal square roots.  Define
\begin{equation}\label{eq:S-star}
    S_*
    =
    \frac12
    \left[
    \kappa\Id+
    \left(\kappa^2\Id+4C_1^{1/2}C_0C_1^{1/2}\right)^{1/2}
    \right]
\end{equation}
and
\begin{equation}\label{eq:L-star}
    L_*=C_1^{-1/2}S_*C_1^{-1/2}.
\end{equation}
Next set
\begin{equation*}
    P_*=[A_0+r_0(\Id-L_*)]^{-1},
    \qquad
    Q_*=L_*P_*L_*.
\end{equation*}
For the mean variables, put
\begin{equation}\label{eq:h-star}
    h_*=(\Id+r_0\Sigma_0+r_1\Sigma_1)^{-1}(m_0-m_1),
\end{equation}
\begin{equation}\label{eq:uv-star}
    u_*=m_0-r_0\Sigma_0h_*,
    \qquad
    v_*=m_1+r_1\Sigma_1h_*.
\end{equation}
Let
\begin{equation*}
    p_* = \cN(u_*,P_*),
    \qquad
    q_* = \cN(v_*,Q_*),
\end{equation*}
and define
\begin{equation}\label{eq:Astar}
    \mathcal A_*
    =W_2^2(p_*,q_*)
    +\tau_0\KL(p_*\mid\mu_0)
    +\tau_1\KL(q_*\mid\mu_1).
\end{equation}

\begin{theorem}[Closed form]\label{thm:main}
The matrices \(P_*\) and \(Q_*\) are positive definite.  The unique minimizer of \eqref{eq:uot-primal} is
\begin{equation*}
    \gamma_* = M_*\pi_*,
\end{equation*}
where
\begin{equation}\label{eq:M-star}
    M_*
    =
    a^{\tau_0/(\tau_0+\tau_1)}
    b^{\tau_1/(\tau_0+\tau_1)}
    \exp\left(-\frac{\mathcal A_*}{\tau_0+\tau_1}\right),
\end{equation}
and \(\pi_*\) is the Gaussian \(W_2\)-optimal coupling between \(p_*\) and \(q_*\).  Equivalently, if \(X\sim p_*\), then
\begin{equation}\label{eq:optimal-map}
    Y=v_*+L_*(X-u_*)
\end{equation}
has law \(q_*\), and \((X,Y)\sim\pi_*\).  The coupling \(\pi_*\) is a generally degenerate Gaussian measure on \(\R^d\times\R^d\), supported on the graph of \eqref{eq:optimal-map}.  The optimal value is
\begin{equation}\label{eq:optimal-value}
    \mathsf U_{\tau_0,\tau_1}(\alpha,\beta)
    =
    \tau_0a+\tau_1b-(\tau_0+\tau_1)M_*.
\end{equation}
\end{theorem}

\section{Proof of the main theorem}

\paragraph{Mass decomposition.}
Let \(M=\gamma(\R^d\times\R^d)\).  For \(M>0\), write \(\gamma=M\pi\), where \(\pi\) is a probability measure with marginals \(p,q\).  Then \(\gamma_0=Mp\) and \(\gamma_1=Mq\).  Since \(\alpha=a\mu_0\) and \(\beta=b\mu_1\),
\begin{equation*}
    \KL(Mp\mid a\mu_0)
    =M\KL(p\mid\mu_0)+M\log\frac{M}{a}-M+a,
\end{equation*}
with the analogous identity for \(\KL(Mq\mid b\mu_1)\).  Minimizing over couplings with fixed marginals gives \(W_2^2(p,q)\), and the problem becomes
\begin{align*}
    \inf_{M\ge0,\,p,q\in\cP_2(\R^d)}
    \Bigl\{&M\mathcal A(p,q)
    +\tau_0\left(M\log\frac{M}{a}-M+a\right) \\
    &\quad +\tau_1\left(M\log\frac{M}{b}-M+b\right)\Bigr\},
\end{align*}
where
\begin{equation}\label{eq:A-pq}
    \mathcal A(p,q)
    =W_2^2(p,q)+\tau_0\KL(p\mid\mu_0)+\tau_1\KL(q\mid\mu_1).
\end{equation}
For fixed \((p,q)\), the minimizer in \(M\) satisfies
\[
    \mathcal A(p,q)+\tau_0\log\frac{M}{a}+\tau_1\log\frac{M}{b}=0,
\]
and therefore
\begin{equation}\label{eq:M-pq}
    M(p,q)
    =a^{\tau_0/(\tau_0+\tau_1)}
     b^{\tau_1/(\tau_0+\tau_1)}
     \exp\left(-\frac{\mathcal A(p,q)}{\tau_0+\tau_1}\right).
\end{equation}
There is a finite competitor, for instance \(p=\mu_0\) and \(q=\mu_1\).  The case \(M=0\) gives the larger value \(\tau_0a+\tau_1b\), while \eqref{eq:M-pq} is positive and yields a smaller value whenever \(\mathcal A(p,q)<\infty\).  It remains to minimize \(\mathcal A(p,q)\).

\paragraph{Reduction to Gaussian marginals.}
The normalized problem admits the following Gaussian reduction.

\begin{lemma}
\label{lem:gaussian-reduction}
Let \(p,q\in\cP_2(\R^d)\) have finite KL terms in \eqref{eq:A-pq}, with means and covariances \((u,P)\) and \((v,Q)\).  Let \(g_0=\cN(u,P)\) and \(g_1=\cN(v,Q)\).  Then \(P,Q\in\Spp^d\) and
\[
    \mathcal A(p,q)\ge \mathcal A(g_0,g_1).
\]
\end{lemma}

\begin{proof}
Since \(p,q\in\cP_2(\R^d)\), their covariance matrices are well defined.  Finite KL with respect to a non-degenerate Gaussian implies \(p\ll\mu_0\) and \(q\ll\mu_1\).  If, for example, \(P\) were singular, then \(p\) would be supported on a proper affine hyperplane, which has \(\mu_0\)-measure zero; this contradicts \(p\ll\mu_0\).  Hence \(P,Q\in\Spp^d\).  Gelbrich's inequality \cite{gelbrich_1990} gives
\begin{equation*}
    W_2^2(p,q)
    \ge
    \|u-v\|^2+\tr\left(P+Q-2(P^{1/2}QP^{1/2})^{1/2}\right)
    =W_2^2(g_0,g_1).
\end{equation*}
The KL terms decrease as well.  We record the argument because it is the finite-dimensional Gaussian projection used in the proof.  Since \(g_0\) and \(\mu_0\) are equivalent Gaussian measures and \(p\ll\mu_0\), the chain rule for relative entropy gives
\begin{equation*}
    \KL(p\mid\mu_0)
    =\KL(p\mid g_0)+\int\log\frac{g_0}{\mu_0}\dd p .
\end{equation*}
The logarithm \(\log(g_0/\mu_0)\) is a quadratic polynomial.  Since \(p\) and \(g_0\) have the same mean and covariance,
\[
    \int\log\frac{g_0}{\mu_0}\dd p
    =
    \int\log\frac{g_0}{\mu_0}\dd g_0
    =
    \KL(g_0\mid\mu_0).
\]
Thus
\[
    \KL(p\mid\mu_0)
    =
    \KL(p\mid g_0)+\KL(g_0\mid\mu_0)
    \ge \KL(g_0\mid\mu_0),
\]
with the usual extended-value interpretation.  The same argument applies to \(q\).
\end{proof}

The remaining minimization is over the Gaussian variables \((u,v,P,Q)\).

For Gaussian measures one has
\begin{align*}
    &\|u-v\|^2+d_B^2(P,Q)
    +\tau_0\KL(\cN(u,P)\mid\mu_0)
    +\tau_1\KL(\cN(v,Q)\mid\mu_1),
\end{align*}
where
\begin{equation*}
    d_B^2(P,Q)=\tr\left(P+Q-2(P^{1/2}QP^{1/2})^{1/2}\right),
\end{equation*}
and
\begin{equation*}
    \KL(\cN(u,P)\mid\cN(m_0,\Sigma_0))
    =\frac12\left[
    \tr(A_0P)+(u-m_0)^TA_0(u-m_0)-d
    +\log\frac{\det\Sigma_0}{\det P}
    \right].
\end{equation*}
The mean and covariance variables separate.

\paragraph{Mean equations.}
The mean component is a strictly convex quadratic problem.  With \(h=u-v\), its first-order equations are
\[
    2(u-v)+\tau_0A_0(u-m_0)=0,
    \qquad
    2(v-u)+\tau_1A_1(v-m_1)=0.
\]
Equivalently,
\[
    u=m_0-r_0\Sigma_0h,
    \qquad
    v=m_1+r_1\Sigma_1h.
\]
Subtracting gives \((\Id+r_0\Sigma_0+r_1\Sigma_1)h=m_0-m_1\), which proves \eqref{eq:h-star} and \eqref{eq:uv-star}.

\paragraph{Covariance equations.}
For \(P,Q\in\Spp^d\), let \(L\in\Spp^d\) be the Gaussian optimal map from \(\cN(0,P)\) to \(\cN(0,Q)\).  Then
\begin{equation*}
    Q=LPL,
\end{equation*}
and \(L^{-1}\) is the Gaussian optimal map in the reverse direction.  The Bures differential formula, see \cite{bhatia_jain_lim_2019,takatsu_2011}, reads
\begin{equation*}
    \nabla_P d_B^2(P,Q)=\Id-L,
    \qquad
    \nabla_Q d_B^2(P,Q)=\Id-L^{-1}.
\end{equation*}
The stationarity equations construct the Gaussian candidate; global optimality is certified independently by the dual argument below.  The covariance first-order equations are
\begin{equation}\label{eq:PQ-foc}
    P^{-1}=A_0+r_0(\Id-L),
    \qquad
    Q^{-1}=A_1+r_1(\Id-L^{-1}).
\end{equation}
Using \(Q^{-1}=L^{-1}P^{-1}L^{-1}\) and multiplying by \(L\) on both sides gives
\begin{equation}\label{eq:riccati}
    LC_1L-\kappa L=C_0.
\end{equation}

To solve \eqref{eq:riccati}, set \(S=C_1^{1/2}LC_1^{1/2}\).  Then
\begin{equation*}
    S^2-\kappa S=B,
    \qquad
    B=C_1^{1/2}C_0C_1^{1/2}\succ0.
\end{equation*}
If \(S\succ0\) solves this equation, then \(B\) is a polynomial in \(S\), so \(S\) and \(B\) commute.  Diagonalizing them simultaneously reduces the equation to \(s^2-\kappa s=\lambda\), \(\lambda>0\), whose unique positive root is
\[
    s=\frac{\kappa+\sqrt{\kappa^2+4\lambda}}{2}.
\]
Functional calculus gives \eqref{eq:S-star}, and then \eqref{eq:L-star}; this is the unique positive definite solution of \eqref{eq:riccati}.

The positivity of \(P_*\) and \(Q_*\) follows from the Riccati equation by the identity
\[
    (\Id-L_*)^2+\frac1{r_0}A_0+\frac1{r_1}L_*A_1L_*
    =
    \frac{r_0+r_1}{r_0r_1}[A_0+r_0(\Id-L_*)].
\]
The left-hand side is positive definite, hence \(P_*^{-1}=A_0+r_0(\Id-L_*)\succ0\), and \(Q_*=L_*P_*L_*\succ0\).  Since \(Q_*^{-1}=L_*^{-1}P_*^{-1}L_*^{-1}\), \eqref{eq:riccati} also gives
\begin{equation}\label{eq:Q-backsub}
    Q_*^{-1}=A_1+r_1(\Id-L_*^{-1}),
\end{equation}
so the pair \((P_*,Q_*)\) satisfies both covariance first-order equations.

\paragraph{Dual certificate and uniqueness.}
A complementary-slackness certificate now completes the argument.

\begin{proposition}
\label{prop:dual-certificate}
The measure \(\gamma_*=M_*\pi_*\) satisfies the KL-unbalanced complementary-slackness conditions.  This proves uniqueness of the minimizer in \eqref{eq:uot-primal}.
\end{proposition}

\begin{proof}
The elementary inequality
\begin{equation}\label{eq:KL-young}
    \tau(z\log z-z+1)+\varphi z
    \ge
    \tau(1-e^{-\varphi/\tau}),
    \qquad z\ge0,
\end{equation}
has equality if and only if \(z=e^{-\varphi/\tau}\).  If measurable \(\phi,\psi\) satisfy
\begin{equation}\label{eq:dual-feasible}
    \phi(x)+\psi(y)\le \|x-y\|^2\qquad\text{for all }x,y,
\end{equation}
then every admissible \(\gamma\) satisfies
\begin{align}\label{eq:weak-dual-bound}
    &\int \|x-y\|^2\dd\gamma
    +\tau_0\KL(\gamma_0\mid\alpha)
    +\tau_1\KL(\gamma_1\mid\beta) \notag\\
    &\qquad\ge
    \tau_0\int\left(1-e^{-\phi/\tau_0}\right)\dd\alpha
    +\tau_1\int\left(1-e^{-\psi/\tau_1}\right)\dd\beta .
\end{align}
If one of the KL terms is \(+\infty\), this inequality is immediate.  Otherwise it follows by applying \eqref{eq:KL-young} to the Radon--Nikodym densities of \(\gamma_0\) and \(\gamma_1\) with respect to \(\alpha\) and \(\beta\), respectively, and using \eqref{eq:dual-feasible}.
Equality in this bound holds if
\begin{equation}\label{eq:kkt-marginals}
    \gamma_0=e^{-\phi/\tau_0}\alpha,
    \qquad
    \gamma_1=e^{-\psi/\tau_1}\beta,
\end{equation}
and  \(\phi(x)+\psi(y)=\|x-y\|^2\) \(\gamma\)-almost everywhere.  The quadratic potentials below have finite Gaussian integrals, so no integrability issue arises.

Define
\begin{equation*}
    \phi(x)=-\tau_0\log\frac{M_*p_*(x)}{a\mu_0(x)},
    \qquad
    \psi(y)=-\tau_1\log\frac{M_*q_*(y)}{b\mu_1(y)},
\end{equation*}
where densities are taken with respect to Lebesgue measure.  Then \eqref{eq:kkt-marginals} holds for \(\gamma_*=M_*\pi_*\).  Let
\[
    X=x-u_*,\qquad Y=y-v_*,\qquad h_*=u_*-v_*.
\]
Using \eqref{eq:PQ-foc}, \eqref{eq:Q-backsub}, and the mean equations, the potentials expand as
\begin{equation*}
    \phi(x)=X^T(\Id-L_*)X+2h_*^TX+c_\phi,
\end{equation*}
\begin{equation*}
    \psi(y)=Y^T(\Id-L_*^{-1})Y-2h_*^TY+c_\psi.
\end{equation*}
The constants satisfy
\begin{equation}\label{eq:constant-sum}
    c_\phi+c_\psi=\|h_*\|^2.
\end{equation}
For completeness we record the short verification of the constant term.  The mass formula \eqref{eq:M-star} is equivalent to
\begin{equation*}
    \tau_0\log\frac{M_*}{a}+\tau_1\log\frac{M_*}{b}+\mathcal A_*=0.
\end{equation*}
After inserting \eqref{eq:Astar} and the Gaussian KL formula into \(c_\phi+c_\psi\), the log-determinant and mean-quadratic terms cancel.  The remaining covariance contribution is zero because the covariance first-order equations imply
\[
    \frac{\tau_0}{2}\{\tr(A_0P_*)-d\}=-\tr(P_*)+\tr(L_*P_*),
    \qquad
    \frac{\tau_1}{2}\{\tr(A_1Q_*)-d\}=-\tr(Q_*)+\tr(L_*P_*),
\]
while \(d_B^2(P_*,Q_*)=\tr(P_*)+\tr(Q_*)-2\tr(L_*P_*)\).  This proves \eqref{eq:constant-sum}.

It follows that
\begin{align*}
    \|x-y\|^2-\phi(x)-\psi(y)
    &=\|h_*+X-Y\|^2-X^T(\Id-L_*)X-Y^T(\Id-L_*^{-1})Y \\
    &\quad -2h_*^T(X-Y)-\|h_*\|^2 \\
    &=X^TL_*X-2X^TY+Y^TL_*^{-1}Y \\
    &=\|L_*^{1/2}X-L_*^{-1/2}Y\|^2\ge0.
\end{align*}
This proves \(\phi+\psi\le\|x-y\|^2\), with equality exactly on the graph \(Y=L_*X\), equivalently on \eqref{eq:optimal-map}.  The marginal identities and graph equality give equality in \eqref{eq:weak-dual-bound} for \(\gamma_*=M_*\pi_*\), proving global optimality.

Let \(\bar\gamma\) be another optimizer.  Equality in \eqref{eq:KL-young} forces its marginals to be \(M_*p_*\) and \(M_*q_*\), while equality in \eqref{eq:dual-feasible} forces \(\bar\gamma\) to be supported on the same graph.  The graph map \(T(x)=v_*+L_*(x-u_*)\) is deterministic, so the only measure with first marginal \(M_*p_*\) supported on this graph is \(M_*(\mathrm{Id},T)_\#p_*\).  This gives \(\bar\gamma=\gamma_*\).
\end{proof}

The mass reduction, Gaussian reduction, mean and covariance equations, and \cref{prop:dual-certificate} together prove \cref{thm:main}.

\section{Consequences and numerical verification}

\begin{remark}
When \(\tau_0=\tau_1=\gamma\), \eqref{eq:uot-primal} is the quadratic-cost KL entropy-transport problem associated with the Gaussian Hellinger--Kantorovich distance \(GHK_\gamma\) \cite{liero_mielke_savare_2018}.  Janati's thesis manuscript records the corresponding Gaussian value as the \(\sigma\downarrow0\) limit of entropic Gaussian UOT \cite[Prop.~14]{janati_thesis_2021}.  The expression above agrees with that symmetric endpoint, while keeping the two marginal relaxation parameters independent.  It also supplies objects that are not explicit from the symmetric value alone: the two modified Gaussian marginals \(M_*p_*,M_*q_*\), the Wasserstein graph minimizer between them, the Riccati description of the covariance map, and the quadratic KL-dual potentials.

The case \(\tau_0\ne\tau_1\) no longer defines a symmetric distance, but it remains a natural convex entropy-transport problem.  Such an asymmetry is useful when creation and destruction relative to the two reference marginals should carry different costs.  It is also distinct from the semi-unbalanced Gaussian setting, where one marginal is fixed while the other is relaxed \cite{nguyen_2024_suot_gaussians}.
\end{remark}

\begin{remark}
For \(\varepsilon>0\), the entropic term \(\varepsilon\KL(\pi\mid\alpha\otimes\beta)\) enforces \(\pi\ll\alpha\otimes\beta\), and the Gaussian optimizer in \cite{janati_2020} is non-degenerate on \(\R^d\times\R^d\).  At the endpoint considered here, the minimizer is \(M_*(\mathrm{Id},T)_\#p_*\) with \(T(x)=v_*+L_*(x-u_*)\), hence is generally singular on the product space.

The distinction is structural.  In the strictly entropic problem the optimality system is a pair of Gaussian Sinkhorn fixed-point equations for scaling functions, which describe a density of the joint plan.  As \(\varepsilon\downarrow0\), that density may concentrate and the joint covariance block degenerates to the rank-\(d\) covariance of a graph-supported Gaussian coupling.  A formal endpoint substitution therefore does not, by itself, prove feasibility and complementary slackness for the unregularized problem.  Here both the singular minimizer and the quadratic dual certificate are constructed at \(\varepsilon=0\).
\end{remark}

\begin{corollary}
\label{cor:large-relaxation}
Let \(\tau_i=\lambda\bar\tau_i\), where \(\bar\tau_i>0\) are fixed, and let \(\lambda\to\infty\).  Set
\[
    \theta_i=\frac{\bar\tau_i}{\bar\tau_0+\bar\tau_1},
    \qquad
    G=a^{\theta_0}b^{\theta_1}.
\]
Then \(p_*\to\mu_0\), \(q_*\to\mu_1\), \(M_*\to G\), and
\begin{align}\label{eq:large-relaxation-expansion}
    \mathsf U_{\lambda\bar\tau_0,\lambda\bar\tau_1}(a\mu_0,b\mu_1)
    &=\lambda\left[\bar\tau_0a+\bar\tau_1b-(\bar\tau_0+\bar\tau_1)G\right] \notag\\
    &\quad +G W_2^2(\mu_0,\mu_1)+o(1).
\end{align}
The coefficient of \(\lambda\) is non-negative and vanishes if and only if \(a=b\).  Consequently, if \(a=b=m\), for arbitrary fixed \(\bar\tau_0,\bar\tau_1>0\),
\[
    \mathsf U_{\lambda\bar\tau_0,\lambda\bar\tau_1}(m\mu_0,m\mu_1)
    \to
    mW_2^2(\mu_0,\mu_1).
\]
\end{corollary}

\begin{proof}
Since \(r_i=2/(\lambda\bar\tau_i)\to0\), the mean formula gives \(u_*\to m_0\) and \(v_*\to m_1\).  The explicit formula for \(L_*\), equivalently the Riccati equation, converges to the unique positive definite solution of \(LA_1L=A_0\), which is the Gaussian \(W_2\) covariance map from \(\Sigma_0\) to \(\Sigma_1\).  By continuity of the principal matrix square root on \(\mathbb S_{++}^d\), this yields \(P_*\to\Sigma_0\), \(Q_*\to\Sigma_1\), and \(\mathcal A_*\to W_2^2(\mu_0,\mu_1)\).

Writing \(\bar T=\bar\tau_0+\bar\tau_1\), the mass formula gives
\[
    M_*
    =
    G\exp\left(-\frac{\mathcal A_*}{\lambda\bar T}\right)
    =
    G\left(1-\frac{\mathcal A_*}{\lambda\bar T}+o(\lambda^{-1})\right),
\]
since \(\mathcal A_*\) has a finite limit.  Substitution into \eqref{eq:optimal-value} yields
\[
    \mathsf U_{\lambda\bar\tau_0,\lambda\bar\tau_1}(a\mu_0,b\mu_1)
    =
    \lambda\left[\bar\tau_0a+\bar\tau_1b-\bar T G\right]
    +G\mathcal A_*+o(1),
\]
and hence \eqref{eq:large-relaxation-expansion}.  Finally,
\[
    \bar\tau_0a+\bar\tau_1b-\bar T G
    =
    \bar T\left(\theta_0a+\theta_1b-a^{\theta_0}b^{\theta_1}\right)\ge0
\]
by the weighted arithmetic-geometric mean inequality, with equality precisely when \(a=b\).
\end{proof}

\begin{remark}
For \(d=1\), write \(\Sigma_i=\sigma_i^2\).  Then
\[
    C_0=\sigma_0^{-2}+r_0,
    \qquad
    C_1=\sigma_1^{-2}+r_1,
\]
and
\begin{equation*}
    L_*=\frac{\kappa+\sqrt{\kappa^2+4C_0C_1}}{2C_1}.
\end{equation*}
The remaining scalar quantities are
\[
    P_*=[\sigma_0^{-2}+r_0(1-L_*)]^{-1},
    \qquad
    Q_*=L_*^2P_*,
\]
and
\[
    h_*=\frac{m_0-m_1}{1+r_0\sigma_0^2+r_1\sigma_1^2},
    \qquad
    u_*=m_0-r_0\sigma_0^2h_*,
    \qquad
    v_*=m_1+r_1\sigma_1^2h_*.
\]
The scalar specialization is useful for numerical evaluation.
\end{remark}

  \Cref{tab:numerics} reports the Riccati residual, the positivity margin of \(P_*^{-1}\), and sampled values of the dual slack.  The residual is
\[
    \|L_*(A_1+r_1\Id)L_*-(r_1-r_0)L_*-(A_0+r_0\Id)\|_F,
\]
where \(\|\cdot\|_F\) is the Frobenius norm.  The dual slack is \(\|x-y\|^2-\phi(x)-\psi(y)\), with equality tested on the supporting graph \(y=v_*+L_*(x-u_*)\).

\begin{table}[H]
	\centering
	\small
	\caption{Algebraic and dual-certificate checks.  The first example is one-dimensional with \(m_0=0.2\), \(\sigma_0=1.1\), \(m_1=1.3\), \(\sigma_1=0.7\), \(\tau_0=1.4\), \(\tau_1=2.2\), \(a=1\), \(b=0.8\).  The second uses a two-dimensional non-commuting covariance pair.}
	\label{tab:numerics}
	\begin{tabular}{lcc}
		\toprule
		Quantity & One-dimensional example & Two-dimensional example \\
		\midrule
		Closed-form value & 0.395206446101 & 0.603839692812 \\
		Optimal mass & 0.767998209416 & 0.843293094695 \\
		Riccati residual Frobenius norm & \(8.88\times10^{-16}\) & \(3.95\times10^{-15}\) \\
		Minimum sampled dual slack & \(7.71\times10^{-8}\) & \(1.04\times10^{-2}\) \\
		Maximum graph equality error & \(2.09\times10^{-14}\) & \(4.44\times10^{-14}\) \\
		Minimum eigenvalue of \(P_*^{-1}\) & 1.125456389751 & 0.837426738528 \\
		\bottomrule
	\end{tabular}
\end{table}

For an independent finite-dimensional comparison, \cref{tab:grid-benchmark} discretizes the one-dimensional example on a uniform grid \(x_1,\ldots,x_n\) spanning six standard deviations around both input Gaussians.  Let \(\alpha_i,\beta_j\ge0\) be the corresponding quadrature weights for the input measures \(\alpha,\beta\).  The discrete KL-UOT dual is
\begin{align}\label{eq:discrete-dual-benchmark}
    \max_{\phi\in\R^n,\,\psi\in\R^n}
    \biggl\{&
    \tau_0\sum_i\alpha_i\left(1-e^{-\phi_i/\tau_0}\right)
    +\tau_1\sum_j\beta_j\left(1-e^{-\psi_j/\tau_1}\right):
    \phi_i+\psi_j\le (x_i-x_j)^2
    \biggr\}.
\end{align}
Here \(\phi_i\) and \(\psi_j\) denote the components of the vectors \(\phi,\psi\in\R^n\).
The values were obtained by a sparse trust-region interior-point solve of \eqref{eq:discrete-dual-benchmark}, using exact first and second derivatives and stopping tolerances \(10^{-8}\).  Since the computation involves truncation and quadrature, \cref{tab:grid-benchmark} is only a reproducibility check; it is not used as a convergence result.

\begin{table}[H]
\centering
\small
\caption{Finite-grid dual benchmark for the one-dimensional example in \cref{tab:numerics}.  The continuous closed-form value is \(0.395206446101\).}
\label{tab:grid-benchmark}
\begin{tabular}{ccccc}
\toprule
Grid size \(n\) & Discrete dual value & Absolute gap & Maximum constraint violation & Iterations \\
\midrule
21 & 0.450038701591 & \(5.48\times10^{-2}\) & \(0\) & 46 \\
31 & 0.417954908724 & \(2.27\times10^{-2}\) & \(0\) & 53 \\
41 & 0.408804277024 & \(1.36\times10^{-2}\) & \(0\) & 76 \\
51 & 0.404012774659 & \(8.81\times10^{-3}\) & \(0\) & 61 \\
\bottomrule
\end{tabular}
\end{table}

\section{Conclusion}

These formulas identify the Gaussian structure of the zero-coupling-entropy problem with independent KL relaxation parameters.  After the total mass is separated, the two penalties select modified Gaussian marginals whose covariance relation is governed by a Riccati equation.  The optimal joint measure is then the Wasserstein graph coupling between them, so singularity is an intrinsic feature of the endpoint problem.

The explicit solution is useful beyond the Gaussian calculation itself.  It gives a continuous test problem in which mass change, marginal relaxation, and classical Gaussian transport are visible in the same set of matrix operations.  For equal total masses, it also recovers the balanced Wasserstein formula when the relaxation parameters become large.  Natural extensions include degenerate Gaussian inputs and related Gaussian barycenter or proximal problems under KL relaxation.  Another open question is to quantify how strictly entropic Gaussian UOT solutions concentrate onto the graph identified above as the coupling entropy tends to zero.





\end{document}